\documentclass[11pt,a4paper,final]{article}
\usepackage{amsmath}%
\numberwithin{equation}{section}
\usepackage{amstext}%
\usepackage{amssymb}%
\usepackage{showkeys}%
\usepackage{epsfig}%
\usepackage{graphicx}%
\usepackage{xcolor}
\usepackage{multicol}
\usepackage{amsthm}
\usepackage{booktabs}
\usepackage{apacite}
\usepackage{multirow}
\usepackage[top=1in, bottom=1.25in, left=0.8in, right=0.8in]{geometry}

\theoremstyle{prop}
\theoremstyle{proof}

\newtheorem{prop}{Proposition}
\newtheorem{theorem}{Theorem}
\newtheorem{cor}{Corollary}
\newtheorem{lemma}[theorem]{Lemma}
\newtheorem{ass}{Assumption}

\providecommand{\keywords}[1]{
\textbf{Keywords:~~~} Partial derivatives, Conditional Expectation, Expected Shortfall.
}

\begin{document}
\title{Derivative Preserving Conditions in Conditional Expectation Operator}
\author{Battulga Gankhuu\footnote{Department of Applied Mathematics, National University of Mongolia; E-mail: battulga.g@seas.num.edu.mn; Phone Number: 976--99246036}}
\date{}

\maketitle 

\begin{abstract}
In this paper, we consider conditions that a higher order derivative preserve in conditional expectation operator for a generic nonlinear random variable. Also, the paper introduces higher order derivatives of the Expected Shortfall for a generic nonlinear portfolio loss random variable.
\end{abstract}


\section{Introduction}

For any integrable random variable $H(x)$, it is clear that $\mathbb{E}[H(x)|H(x)=0]=0$. It is the well--known fact that a derivative of risk measure Value--at--Risk (VaR), which is commonly used in practice equals conditional expectation of the derivative of linear portfolio loss random variable given that the portfolio loss random variable is equal to VaR, see \citeA{Tasche99} and \citeA{Tasche00}. Therefore, if we take $H(x)=L(x)-q_\alpha(L(x))$, then we have
\begin{equation}\label{0.01}
\mathbb{E}\bigg[\frac{\partial }{\partial x_i}H(x)\bigg|H(x)=0\bigg]=0,
\end{equation}
where $L(x)$ is a linear portfolio loss random variable, depending on a weight vector $x$ and $q_\alpha(L(x))$ is the VaR at confidence level $\alpha\in(0,1)$. Thus, for the risk measure Value--at--Risk, the partial derivative preserve in the conditional expectation operator that depends on weight vector $x$. Then, a question arises: under what conditions, a higher order derivative preserve in conditional expectation operator, depending on the weight vector $x$. In section 2, we aim to answer this question for a generic (nonlinear) integrable random variable. 

\citeA{Artzner99} suppose that every risk measure should satisfy four axioms, namely, monotonicity, translation invariance, subadditivity, and positive homogeneity. They refer to a risk measure, satisfying the axioms as a coherent risk measure. For linear portfolio loss random variable, VaR satisfies monotonicity, translation invariance, and positive homogeneity but does not satisfy subadditivity in general. However, Expected Shortfall (ES) is a coherent risk measure for linear portfolio loss random variable, see \citeA{Acerbi02}. Section 3 is dedicated to higher order derivatives of the ES for generic nonlinear integrable portfolio loss random variable. In section 4, we summarize the results of the paper.

\section{Derivatives of Conditional Expectation}

Let $H:\Omega\times\mathbb{R}^d\to \mathbb{R}$ be an integrable random variable, which is defined on probability space $(\Omega,\mathcal{F},\mathbb{P})$. Here we assume that the random variable $H(x)$ is sufficiently differentiable with respect to argument $x=(x_1,\dots,x_d)'\in\mathbb{R}^d$. 

The following Lemma will be used to prove Propositions 1 and 2.
 
\begin{lemma}\label{l.01}
Let $H_1,H_2,\dots$ be a random sequence such that $|H_m|\leq Z$ for all $m=1,2,\dots$, where $Z\geq 0$ is an integrable random variable. If we assume that $\{A_m\}$ is a monotone increasing sequence such that $\lim_{m\to\infty}\mathbb{P}[A_m]=0$, then we have that
\begin{equation}\label{2.01}
\lim_{m\to\infty}\mathbb{E}\big[|H_m|1_{A_m}\big]=0
\end{equation}
and
\begin{equation}\label{2.48}
\lim_{m\to\infty}\mathbb{E}\big[H_m1_{A_m}\big]=0.
\end{equation}
\end{lemma}

\begin{proof}
Since for all $m=1,2,\dots$, $|H_m|$ is bounded by $Z$, we get the following inequality
\begin{equation}\label{2.02}
\lim_{m\to\infty}\mathbb{E}\big[|H_m|1_{A_m}\big]\leq\lim_{m\to\infty}\mathbb{E}\big[Z1_{A_m}\big].
\end{equation}
To prove the Lemma, we use a method, which is referred to standard machine, see \citeA{Shreve04}. Firstly, let us assume that $Z$ equals an any indicator function for a generic set $B\in\mathcal{F}$, that is, $Z=1_{B}$. Then, 
\begin{equation}\label{2.03}
\lim_{m\to\infty}\mathbb{E}\big[Z1_{A_m}\big]=\lim_{m\to\infty}\mathbb{E}\big[1_{B}1_{A_m}\big]=\lim_{m\to\infty}\mathbb{P}\big[BA_m\big]\leq \lim_{m\to\infty}\mathbb{P}\big[A_m\big]=0.
\end{equation}
Secondly, let us assume that $Z$ is an any simple random variable, that is, $Z=s:=\sum_{i=1}^ns_{i}1_{B_{i}}$, where for $i=1,\dots,n$, $s_{i}\geq 0$, for $i,j=1,\dots,n$ and $i\neq j$, $B_{i}\cap B_{j}=\O$, and $\bigcup_{i=1}^nB_{i}=\Omega$. Then, we get that
\begin{equation}\label{2.04}
\lim_{m\to\infty}\mathbb{E}\big[Z1_{A_m}\big]=\lim_{m\to\infty}\mathbb{E}\big[s1_{A_m}\big]=\sum_{i=1}^ns_{i}\lim_{m\to\infty}\mathbb{E}\big[1_{B_{i}}1_{A_m}\big]=0.
\end{equation}
Thirdly, let us assume that $Z$ is an any nonnegative random variable. Then, there is a monotone increasing sequence of simple random variables that converges to $Z$, i.e., $s_{k}\uparrow Z$, as $k\to\infty$, where for $k=1,2,\dots$, $s_{k}$ are simple random variables, see \citeA{Klenke13} and \citeA{Kopp14}. Thus, as $A_m$ is the monotone increasing sequence, $\mathbb{E}\big[s_{k}1_{A_m}\big]$ is a monotone increasing sequence with respect to both index $m$ and $k$. Consequently, according to the monotone convergence theorem, we obtain that
\begin{equation}\label{2.05}
\lim_{m\to\infty}\mathbb{E}\big[Z1_{A_m}\big]=\lim_{m\to\infty}\Big(\lim_{k\to\infty}\mathbb{E}\big[s_{k}1_{A_m}\big]\Big)=\lim_{k\to\infty}\Big(\lim_{m\to\infty}\mathbb{E}\big[s_{k}1_{A_m}\big]\Big)=0.
\end{equation}
Finally, let us assume that $Z$ is an any real random variable. Then, the random variable $Z$ can be written as $Z=Z^+-Z^-$, where $Z^+:=\max(0,Z)$ and $Z^-:=\max(0,-Z)$ are the positive and negative parts of the random variable $Z$. Consequently, since the positive and negative parts of $Z$ are nonnegative, by equation \eqref{2.05}, we have that
\begin{equation}\label{2.06}
\lim_{m\to\infty}\mathbb{E}\big[Z1_{A_m}\big]=\lim_{m\to\infty}\mathbb{E}\big[Z^+1_{A_m}\big]-\lim_{m\to\infty}\mathbb{E}\big[Z^-1_{A_m}\big]=0.
\end{equation}
As a result, according to equation \eqref{2.02}, one obtains equation \eqref{2.01}. To prove \eqref{2.48}, observe that for $m=1,2,\dots$, $H_m=H_m^+-H_m^-$ and $|H_m|=H_m^++H_m^-$, where $H_m^+:=\max(0,H_m)$ and $H_m^-:=\max(0,-H_m)$ are the positive and negative parts of the random variable $H_m$. Since the positive and negative parts of $H_m$ are nonnegative and integrable and the indicator random variable $1_{A_m}$ is nonnegative, it follows from equation \eqref{2.01} that
\begin{equation}\label{2.07}
\lim_{m\to\infty}\mathbb{E}\big[H_m^+1_{A_m}\big]=\lim_{m\to\infty}\mathbb{E}\big[H_m^-1_{A_m}\big]=0.
\end{equation}
Consequently, one obtains equation \eqref{2.48}.
\end{proof}

Here we consider Proposition, which tackles higher order derivatives of expectation of a generic (nonlinear) random variable.
\begin{prop}\label{t.01}
Let $I\subset \mathbb{R}^d$ be an open set, for all $x\in I$, $H(x)$ be an integrable random variable and its $n$--th order partial derivative $\frac{\partial^n}{\partial x_i^n}H(x)$ is continuous with respect to $i$--th ($i=1,\dots,d$) argument, and $Z:=\sup_{x\in I}\big|\frac{\partial^n}{\partial x_i^n}H(x)\big|$ be an integrable random variable. If on a set $\{H(x)=0\}$, $H(x)$ or $|H(x)|$ is a monotone decreasing function with respect to $i$--th ($i=1,\dots,d$) argument, then it holds
\begin{equation}\label{2.08}
\mathbb{E}\bigg[\bigg|\frac{\partial^n}{\partial x_i^n}H(x)\bigg|1_{\{H(x)=0\}}\bigg]=0,
\end{equation}
\begin{equation}\label{2.09}
\mathbb{P}\bigg[\bigg\{\frac{\partial^n}{\partial x_i^n}H(x)=0\bigg\}\bigcap\{H(x)=0\}\bigg]=\mathbb{P}\big[H(x)=0\big],
\end{equation}
and
\begin{equation}\label{2.10}
\mathbb{E}\bigg[\bigg(\frac{\partial^n}{\partial x_i^n}H(x)\bigg)1_{\{H(x)=0\}}\bigg]=0.
\end{equation}
\end{prop}

\begin{proof}
Let us fix $x\in I$, $r=0,\dots,n$ and $i=1,\dots,d$. We suppose that $x+\frac{r}{m}e_i \in I$ for all $m=1,2,\dots$ and $r=1,\dots,n$, where $e_i\in \mathbb{R}^d$ is a unit vector, whose $i$--th component is one and others are zero. ($i$) Let us assume that on the set $\{H(x)=0\}$, $H(x)$ is a monotone decreasing with respect to $i$--th argument $x_i$. Since $H(x)$ continuous with respect to the argument $x$, $\big|H\big(x+\frac{r}{m}e_i\big)\big|\leq |H(x)|$ for $m=1,2,\dots$, and $\lim_{m\to\infty}\big\{H(x+\frac{r}{m}e_i)\geq \frac{1}{k}\big\}=\bigcup_{m=1}^\infty\big\{H(x+\frac{r}{m}e_i)\geq \frac{1}{k}\big\}=\big\{H(x)> \frac{1}{k}\big\}$, due to the dominated convergence theorem, it is clear that 
\begin{equation}\label{2.11}
\lim_{m\to\infty} \mathbb{E}\bigg[H\bigg(x+\frac{r}{m}e_i\bigg)1_{\{H(x+\frac{r}{m}e_i)\geq \frac{1}{k}\}}1_{\{H(x)=0\}}\bigg]=\mathbb{E}\big[H(x)1_{\O}\big]=0
\end{equation}
for all $k=1,2,\dots$. Let $B(x):=\{H(x)=0\}$ and $B_{m,k}^{i,r}(x):=\big\{H\big(x+\frac{r}{m}e_i\big)\geq \frac{1}{k}\big\}$ for $m,k=1,2,\dots$. Then, we have
\begin{equation}\label{2.12}
B_m^{i,r}(x):=\bigg\{H\bigg(x+\frac{r}{m}e_i\bigg)>0\bigg\}=\bigcup_{k=1}^\infty B_{m,k}^{i,r}(x).
\end{equation}
As $\{B_{m,k}^{i,r}(x)\}$ is a monotone increasing sequence with respect to the index $k$, it holds
\begin{equation}\label{2.13}
\mathbb{P}\big[B_m^{i,r}(x)\big]=\lim_{k\to\infty}\mathbb{P}\big[B_{m,k}^{i,r}(x)\big]
\end{equation}
Thus, we have
\begin{equation}\label{2.14}
\mathbb{P}\big[B(x)\cap B_m^{i,r}(x)\big]=\lim_{k\to\infty}\mathbb{P}\big[B(x)\cap B_{m,k}^{i,r}(x) \big].
\end{equation}
Let us define a simple function $s_{m,k}(x):=\frac{1}{k}1_{B(x)\cap B_{m,k}^{i,r}(x)}$. Then, the following inequality is true
\begin{equation}\label{2.15}
s_{m,k}(x)=\frac{1}{k}1_{B_{m,k}^{i,r}(x)}1_{B(x)}\leq H\bigg(x+\frac{r}{m}e_i\bigg)1_{B_{m,k}^{i,r}(x)}1_{B(x)}.
\end{equation}
Therefore, one gets that
\begin{equation}\label{2.16}
\frac{1}{k}\mathbb{P}\big[B(x)\cap B_{m,k}^{i,r}(x)\big]=\int_{\Omega}s_{m,k}(x)d\mathbb{P}\leq\mathbb{E}\bigg[H\bigg(x+\frac{r}{m}e_i\bigg)1_{B_{m,k}^{i,r}(x)}1_{B(x)}\bigg].
\end{equation}
As a result, it follows from equations \eqref{2.11} and \eqref{2.16} that
\begin{equation}\label{2.17}
\lim_{m\to\infty} \mathbb{P}\big[B_{m,k}^{i,r}(x)\cap B(x)\big]=0.
\end{equation}
On the other hand, since on the set $B(x)$, the random variable $H(x)$ is monotone decreasing function, we have
\begin{equation}\label{2.18}
\bigg[H\bigg(x+\frac{r}{m+1}e_i\bigg)-H\bigg(x+\frac{r}{m}e_i\bigg)\bigg]1_{B(x)}\geq 0.
\end{equation}
Thus, the set sequence $\{B(x)\cap B_{m,k}^{i,r}(x)\}$ is monotone increasing with respect to the both indexes $m$ and $k$. Therefore, it follows from equations \eqref{2.14} and \eqref{2.17} that 
\begin{equation}\label{2.19}
\lim_{m\to\infty}\mathbb{P}\big[B(x)\cap B_m^{i,r}(x)\big]=\lim_{m\to\infty}\Big(\lim_{k\to\infty}\mathbb{P}\big[B(x)\cap B_{m,k}^{i,r}(x)\big]\Big)=\lim_{k\to\infty}\Big(\lim_{m\to\infty}\mathbb{P}\big[B(x)\cap B_{m,k}^{i,r}(x)\big]\Big)=0.
\end{equation}

($ii$) Now, we assume that $|H(x)|$ is a monotone decreasing with respect to $i$--th argument $x_i$.
Since on the set $B(x)$, $\big|H\big(x+\frac{r}{m}e_i\big)\big|\leq |H(x)|$ for $m=1,2,\dots$ and $H(x)$ is the continuous function, it follows from dominated convergence theorem that 
\begin{equation}\label{2.46}
\lim_{m\to\infty} \mathbb{E}\bigg[\bigg|H\bigg(x+\frac{r}{m}e_i\bigg)\bigg|1_{\{H(x)=0\}}\bigg]=0.
\end{equation}
If we define $B_{m,k}^{i,r}(x)=\big\{\big|H\big(x+\frac{r}{m}e_i\big)\big|\geq \frac{1}{k}\big\}$ for $m,k=1,2,\dots$, then as $B_{m,k}^{i,r}(x)$ is monotone increasing with respect to the index $k$, we have 
\begin{equation}\label{2.47}
B_m^{i,r}(x):=\bigg\{\bigg|H\bigg(x+\frac{r}{m}e_i\bigg)\bigg|>0\bigg\}=\bigcup_{k=1}^\infty B_{m,k}^{i,r}(x).
\end{equation}
By replacing $H(x)$, which appears in ($i$) part of the Proposition with $|H(x)|$, it can be shown that equation \eqref{2.19} still holds for the monotone decreasing $|H(x)|$. It should be noted that the rest of the proof is exactly the same for both cases. 

Let us introduce the following notations: for fixed $n=1,2,\dots$ and arbitrary $m=1,2,\dots$, $A_{m}^{i,n}(x):=\bigcap_{r=1}^n \overline{B_{m,m^{n+1}}^{i,r}(x)}$ and
\begin{equation}\label{2.20}
\Delta_m^{i,n} H(x):=\sum_{j=0}^n(-1)^{n-j}C_n^jH\bigg(x+\frac{j}{m}e_i\bigg),
\end{equation}
where $C_n^i:=\frac{n!}{i!(n-i)!}$ is the binomial coefficient and for a generic set $A\in \mathcal{F}$, $\overline{A}=\Omega\backslash A$ is the complement of the set $A$. If we repeat the mean value theorem of calculus, then there is a sequence $\{x_m\}\in I$ such that
\begin{equation}\label{2.21}
\frac{\partial^n}{\partial x_i^n}H(x_m)= m^n\Delta_m^{i,n}H(x)~~~\text{and}~~~\lim_{m\to\infty}x_m=x.
\end{equation}
Consequently, $\big|m^n\Delta_m^{i,n}H(x)\big|\leq Z$. Since on the set $B(x)\cap A_m^{i,n}(x)$, $\big|\Delta_m^{i,n} H(x)\big|< \frac{2^n}{m^{n+1}}$, the following inequality is true
\begin{equation}\label{2.22}
\mathbb{E}\big[\big|m^n\Delta_m^{i,n}H(x)\big|1_{B(x)}\big]< \frac{2^n}{m}+\mathbb{E}\bigg[\Big|m^n\Delta_m^{i,n}H(x)\Big|1_{\big\{B(x)\cap \overline{A_{m}^{i,n}(x)}\big\}}\bigg].
\end{equation}
Observe that $\{B(x)\cap \overline{A_m^{i,n}(x)}\}$ is monotone increasing sequence with respect to index $m$. Due to equations \eqref{2.19} and \eqref{2.12} or \eqref{2.47}, we conclude that
\begin{equation}\label{2.23}
\lim_{m\to\infty}\mathbb{P}\Big[B(x)\cap \overline{A_{m}^{i,n}(x)}\Big]\leq \lim_{m\to\infty}\sum_{r=1}^n\mathbb{P}\Big[B(x)\cap B_{m,m^{n+1}}^{i,r}(x)\Big]\leq \lim_{m\to\infty}n\mathbb{P}\Big[B(x)\cap B_{m}^{i,r}(x)\Big]=0.
\end{equation}
Therefore, it follows from equation \eqref{2.22} and \eqref{2.23}, the dominated convergence theorem, and Lemma \ref{l.01} that
\begin{equation}\label{2.24}
\lim_{m\to\infty}\mathbb{E}\big[\big|m^n\Delta_m^{i,n}H(x)\big|1_{B(x)}\big]=\mathbb{E}\bigg[\bigg|\frac{\partial^n}{\partial x^n} H(x)\bigg|1_{B(x)}\bigg]=0.
\end{equation}

Equation \eqref{2.24} implies that
\begin{equation}\label{2.27}
\mathbb{P}\bigg[\bigg(\frac{\partial^n}{\partial x^n} H(x)\bigg)1_{B(x)}= 0\bigg]=1.
\end{equation}
Thus, it holds
\begin{equation}\label{2.28}
1=\mathbb{P}\bigg[\bigg\{\bigg(\frac{\partial^n}{\partial x^n} H(x)\bigg)1_{B(x)}= 0\bigg\}\bigcap B(x)\bigg]+\mathbb{P}\bigg[\bigg\{\bigg(\frac{\partial^n}{\partial x^n} H(x)\bigg)1_{B(x)}= 0\bigg\}\bigcap \overline{B(x)}\bigg].
\end{equation}
As a result, equation \eqref{2.09} holds. 

Let us denote for a generic random variable $X$, its positive and negative parts by $X^+$ and $X^-$, respectively, that is, $X^+:=\max(0,X)$ and $X^-:=\max(0,-X)$. Then, as the proof of Lemma 1, $|X|=X^++X^-$ and $X=X^+-X^-$. According to equation \eqref{2.24} and the fact that the positive and negative parts are nonnegative, we have
\begin{equation}\label{2.29}
\mathbb{E}\bigg[\bigg(\frac{\partial^n}{\partial x^n} H(x)\bigg)^+1_{B(x)}\bigg]=\mathbb{E}\bigg[\bigg(\frac{\partial^n}{\partial x^n} H(x)\bigg)^-1_{B(x)}\bigg]=0.
\end{equation}
Consequently, it holds
\begin{equation}\label{2.30}
\mathbb{E}\bigg[\bigg(\frac{\partial^n}{\partial x^n} H(x)\bigg)1_{B(x)}\bigg]=0.
\end{equation}
This completes the proof of the Proposition.
\end{proof}

The following two Corollaries, dealing with higher order derivatives of discrete and absolute continuous portfolio loss random variables are immediate consequences of Proposition 1.

\begin{cor}\label{c.01}
Let $I\subset \mathbb{R}^d$ be an open set, for all $x\in I$, $H(x)$ be an integrable random variable and its $n$--th order partial derivative $\frac{\partial^n}{\partial x_i^n}H(x)$ is continuous with respect to $i$--th ($i=1,\dots,d$) argument, and $Z:=\sup_{x\in I}\big|\frac{\partial^n}{\partial x_i^n}H(x)\big|$ be an integrable random variable. If $\mathbb{P}\big[H(x)=0\big]>0$ and on a set $\{H(x)=0\}$, $H(x)$ or $|H(x)|$ is a monotone decreasing function with respect to $i$--th ($i=1,\dots,d$) argument, then it holds
\begin{equation}\label{2.31}
\mathbb{E}\bigg[\bigg|\frac{\partial^n}{\partial x_i^n}H(x)\bigg|\bigg|H(x)=0\bigg]=0,
\end{equation}
\begin{equation}\label{2.32}
\mathbb{P}\bigg[\frac{\partial^n}{\partial x_i^n}H(x)=0 \bigg|H(x)=0\bigg]=1,
\end{equation}
and
\begin{equation}\label{2.33}
\mathbb{E}\bigg[\frac{\partial^n}{\partial x_i^n}H(x)\bigg|H(x)=0\bigg]=0.
\end{equation}
\end{cor}
\begin{proof}
Since $\mathbb{P}\big[H(x)=0\big]>0$, if we divide the both sides of equations \eqref{2.08}--\eqref{2.10} by $\mathbb{P}\big[H(x)=0\big]$, then the proof of the Corollary holds.
\end{proof}

\begin{cor}\label{c.02}
Let $I\subset \mathbb{R}^d$ be an open set, for all $x\in I$, $H(x)$ be an absolute continuous integrable random variable and its $n$--th order partial derivative $\frac{\partial^n}{\partial x_i^n}H(x)$ is continuous with respect to $i$--th ($i=1,\dots,d$) argument, and $Z:=\sup_{x\in I}\big|\frac{\partial^n}{\partial x_i^n}H(x)\big|$ be an integrable random variable. If on a set $\{H(x)=0\}$, $H(x)$ or $|H(x)|$ is a monotone decreasing function with respect to $i$--th ($i=1,\dots,d$) argument, then it holds
\begin{equation}\label{2.34}
\mathbb{E}\bigg[\bigg|\frac{\partial^n}{\partial x_i^n}H(x)\bigg|\bigg|H(x)=0\bigg]=0,
\end{equation}
\begin{equation}\label{2.35}
\mathbb{P}\bigg[\frac{\partial^n}{\partial x_i^n}H(x)=0 \bigg|H(x)=0\bigg]=1,
\end{equation}
and
\begin{equation}\label{2.36}
\mathbb{E}\bigg[\frac{\partial^n}{\partial x_i^n}H(x)\bigg|H(x)=0\bigg]=0.
\end{equation}
\end{cor}

\begin{proof}
Because $\big\{0\leq H(x)\leq \frac{1}{m}\big\}$ is a monotone decreasing sequence and its limit is $\{H(x)=0\}$, by monotone convergence theorem for decreasing sequence, equation \eqref{2.08} can be written by
\begin{equation}\label{2.37}
\mathbb{E}\bigg[\bigg|\frac{\partial^n}{\partial x_i^n}H(x)\bigg|1_{\{H(x)=0\}}\bigg]=\lim_{\varepsilon\to 0}\mathbb{E}\bigg[\bigg|\frac{\partial^n}{\partial x_i^n}H(x)\bigg|1_{\{0\leq H(x)\leq \varepsilon\}}\bigg]=0.
\end{equation}
Consequently, one has
\begin{eqnarray}
&&\lim_{\varepsilon\to 0}\frac{\mathbb{E}\big[\big|\frac{\partial^n}{\partial x_i^n}H(x)\big|1_{\{0\leq H(x)\leq 2\varepsilon\}}\big]}{\varepsilon^2}\nonumber\\
&&=\lim_{\varepsilon\to 0}\frac{\mathbb{E}\big[\big|\frac{\partial^n}{\partial x_i^n}H(x)\big|1_{\{0\leq H(x)\leq 2\varepsilon\}}\big]-2\mathbb{E}\big[\big|\frac{\partial^n}{\partial x_i^n}H(x)\big|1_{\{0\leq H(x)\leq \varepsilon\}}\big]}{\varepsilon^2}\label{2.38}\\
&&=\frac{\partial^2 }{\partial t^2}\mathbb{E}\bigg[\bigg|\frac{\partial^n}{\partial x_i^n}H(x)\bigg|1_{\{H(x)\leq t\}}\bigg]\bigg|_{t=0}.\nonumber
\end{eqnarray}
Since the random variable $H(x)$ is absolute continuous, we have 
\begin{equation}\label{2.39}
\lim_{\varepsilon\to 0}\frac{\mathbb{P}\big[0\leq H(x)\leq 2\varepsilon\big]}{\varepsilon}=2f_{H(x)}(0),
\end{equation}
where $f_{H(x)}(t)$ is a density function of the random variable $H(x)$.
As a result, we find that
\begin{equation}\label{2.40}
\mathbb{E}\bigg[\bigg|\frac{\partial^n}{\partial x_i^n}H(x)\bigg|\bigg|H(x)=0\bigg]=\lim_{\varepsilon\to 0}\frac{\mathbb{E}\big[\big|\frac{\partial^n}{\partial x_i^n}H(x)\big|1_{\{0\leq H(x)\leq 2\varepsilon\}}\big]}{\mathbb{P}\big[0\leq H(x)\leq 2\varepsilon\big]}=0.
\end{equation}
Following the proof of Theorem \ref{t.01}, one can prove equations \eqref{2.35} and \eqref{2.36}. Thus, we prove the Corollary.
\end{proof}

To relax the condition that on the set $\{H(x)=0\}$ the random variable $H(x)$ or $|H(x)|$ is monotone decreasing with respect to its $i$-th $(i=1,\dots,d)$ argument, we assume that the following assumption holds.
\begin{ass}\label{a.01}
Let us assume that for integrable random variables $\frac{\partial}{\partial x_i} H(x)$, $i=1,\dots,d$, the following condition holds
\begin{equation}\label{2.41}
\frac{\partial}{\partial x_i}\mathbb{E}\bigg[\frac{\partial H(x)}{\partial x_j}\bigg|H(x)=0\bigg]=\frac{\partial}{\partial x_j}\mathbb{E}\bigg[\frac{\partial H(x)}{\partial x_i}\bigg|H(x)=0\bigg]~~~\text{for all}~i,j=1,\dots,d,
\end{equation}
\end{ass}
If Assumption \ref{a.01} holds, then the following Proposition is true.
\begin{prop}\label{p.01}
Let for all $x\in I$, $H(x)$ be an integrable random variable and positive homogeneous function with respect to argument $x$, where $I\subset \mathbb{R}^d$ is an open set. If Assumption \ref{a.01} is true, then it holds
\begin{equation}\label{2.42}
\mathbb{E}\bigg[\frac{\partial H(x)}{\partial x_i}\bigg|H(x)=0\bigg]=0~~~\text{for}~i=1,\dots,d.
\end{equation}
\end{prop}

\begin{proof}
According to the Euler's formula, it holds
\begin{equation}\label{2.43}
H(x)=\sum_{i=1}^n \frac{\partial H(x)}{\partial x_i}x_i.
\end{equation}
Since $\mathbb{E}\big[H(x)\big|H(x)=0\big]=0$, we have 
\begin{eqnarray}
0&=& \frac{\partial}{\partial x_i}\mathbb{E}\big[H(x)\big|H(x)=0\big]=\frac{\partial}{\partial x_i}\bigg\{\sum_{j=1}^n\mathbb{E}\bigg[\frac{\partial H(x)}{\partial x_j}\bigg|H(x)=0\bigg]x_j\bigg\}\nonumber\\
&=&\sum_{j=1}^n\bigg(\frac{\partial}{\partial x_i}\mathbb{E}\bigg[\frac{\partial H(x)}{\partial x_j}\bigg|H(x)=0\bigg]\bigg)x_j+\mathbb{E}\bigg[\frac{\partial H(x)}{\partial x_i}\bigg|H(x)=0\bigg].\label{2.44}
\end{eqnarray}
As $\mathbb{E}\big[\frac{\partial H(x)}{\partial x_i}\big|H(x)=0\big]$ is a homogeneous function with degree of zero, it holds 
\begin{equation}\label{2.45}
\sum_{j=1}^n\bigg(\frac{\partial}{\partial x_j}\mathbb{E}\bigg[\frac{\partial H(x)}{\partial x_i}\bigg|H(x)=0\bigg]\bigg)x_j=0.
\end{equation}
Since Assumption \ref{a.01} holds, it follows from equations \eqref{2.44} and \eqref{2.45} that $\mathbb{E}\big[\frac{\partial H(x)}{\partial x_i}\big|H(x)=0\big]=0$. That completes the proof.
\end{proof}

\section{Derivatives of Expected Shortfall}

Let $L:\Omega\times\mathbb{R}^d\to \mathbb{R}$ be a nonlinear portfolio loss random variable, defined on the probability space $(\Omega,\mathcal{F},\mathbb{P})$ and for a confidence level $\alpha\in(0,1)$, $q_\alpha(x):=\inf\big\{t\in \mathbb{R}\big| F_{L(x)}(t)\geq \alpha\big\}$ be an $\alpha$--quantile (VaR) of the portfolio loss random variable $L(x)$, where $F_{L(x)}(t)$ is a distribution function of $L(x)$. Here we assume that the random variable $L(x)$ and the risk measure $q_\alpha(x)$ are sufficiently differentiable with respect to argument $x=(x_1,\dots,x_d)'\in\mathbb{R}^d$. 

In order to obtain higher order derivatives of the Expected Shortfall, we will follow ideas in \citeA{Acerbi02}. Let $U$ be an uniform random variable on $[0,1]$ and for $\alpha\in (0,1)$, $i=1,\dots,d$, $r=0,1,\dots,n$, and $m=1,2,\dots$, $q_\alpha\big(x+\frac{r}{m}e_i\big)$ be an $\alpha$--quantile of the loss random variable $L\big(x+\frac{r}{m}e_i\big)$. To keep notations simple, we introduce $H(x):=q_U(x)-q_\alpha(x)$. Then, since the quantile function $q_{\alpha}(x)$ is a monotone increasing function with respect to the parameter $\alpha$, it holds
\begin{equation}\label{3.01}
\{U>\alpha\}\subset \bigg\{H\bigg(x+\frac{r}{m}e_i\bigg)\geq 0\bigg\}
~~~\text{and}~~~\{U\leq\alpha\}\subset \bigg\{H\bigg(x+\frac{r}{m}e_i\bigg)\leq 0\bigg\}.
\end{equation}
Consequently, one gets that
\begin{equation}\label{3.02}
\{U\leq \alpha\}\bigcap \bigg\{H\bigg(x+\frac{r}{m}e_i\bigg)\geq 0\bigg\}\subset\bigg\{H\bigg(x+\frac{r}{m}e_i\bigg)= 0\bigg\}
\end{equation}
for $r=0,1,\dots,n$. Note that by the quantile transformation lemma (see \citeA{McNeil15}), for $r=0,1,\dots,n$,  
\begin{equation}\label{3.03}
\bigg\{H\bigg(x+\frac{r}{m}e_i\bigg)=0\bigg\}\overset{d}{=}\bigg\{L\bigg(x+\frac{r}{m}e_i\bigg)=q_\alpha\bigg(x+\frac{r}{m}e_i\bigg)\bigg\},
\end{equation}
where $d$ means equal distribution. Let us denote the ES at confidence level $\alpha\in(0,1)$ of the portfolio loss random variable $L(x)$ by $\text{ES}_\alpha(x)$. Then, the following Proposition holds.

\begin{prop}\label{t.02}
Let $I\subset \mathbb{R}^d$ be an open set, for all $x\in I$, $H(x):=L(x)-q_\alpha(x)$ be an integrable random variable and its $n$--th order partial derivative $\frac{\partial^n}{\partial x_i^n}H(x)$ is continuous with respect to $i$--th ($i=1,\dots,d$) argument, and $Z:=\sup_{x\in I}\big|\frac{\partial^n}{\partial x_i^n}L(x)\big|$ be an integrable random variable. If on a set $C(x)$, $H(x)$ is a monotone increasing function with respect to $i$--th ($i=1,\dots,d$) argument $x_i$, then it holds
\begin{equation}\label{3.04}
\frac{\partial^n}{\partial x_i^n}\mathrm{ES}_\alpha(x)=\frac{1}{1-\alpha}\bigg(\mathbb{E}\bigg[\frac{\partial^n}{\partial x_i^n}L(x)1_{\{L(x)\geq q_\alpha(x)\}}\bigg]-\frac{\partial^n}{\partial x_i^n}q_\alpha(x)\Big(\mathbb{P}\big[L(x)\geq q_\alpha(x)\big]-(1-\alpha)\Big)\bigg),
\end{equation}
where $C(x):=\{U\leq \alpha\}\cup \{H(x)\geq 0\}$.
\end{prop}

\begin{proof}
Let us fix $i=1,\dots,d$. According to the definition of the Expected Shortfall, we have
\begin{equation}\label{3.05}
\text{ES}_\alpha(x)=\frac{1}{1-\alpha}\int_\alpha^1q_u(x)du,
\end{equation}
see \citeA{Acerbi02}. Therefore, by equation \eqref{3.01} (for $r=0$) we have
\begin{eqnarray}
&&(1-\alpha)m^n\Delta_m^{i,n}\text{ES}_\alpha(x)=\int_\alpha^1m^n\Delta_m^{i,n}q_u(x)du=\mathbb{E}\Big[m^n\Delta_m^{i,n}q_U(x)1_{\{U>\alpha\}}\Big]\nonumber\\
&&=\mathbb{E}\Big[m^n\Delta_m^{i,n}q_U(x)1_{\{H(x)\geq 0\}}\Big]-\mathbb{E}\Big[m^n\Delta_m^{i,n}q_U(x)1_{C(x)}\Big]\label{3.06}
\end{eqnarray}
According to the mean value theorem of calculus, there is a sequence $\{x_m\}\in I$ such that
\begin{equation}\label{3.07}
\frac{\partial^n}{\partial x_i^n}L(x_m)= m^n\Delta_m^{i,n}L(x)~~~\text{and}~~~\lim_{m\to\infty}x_m=x.
\end{equation}
Consequently, $\big|m^n\Delta_m^{i,n}L(x)\big|\leq Z$. For the first term of the right--hand side of equation \eqref{3.06}, according to equation \eqref{3.03}, one obtains that
\begin{equation}\label{3.08}
\mathbb{E}\Big[m^n\Delta_m^{i,n}q_U(x)1_{\{H(x)\geq 0\}}\Big]=\mathbb{E}\Big[m^n\Delta_m^{i,n}L(x)1_{\{L(x)\geq q_\alpha(x)\}}\Big]
\end{equation}
By the dominated convergence theorem, equation \eqref{3.08} converges to
\begin{equation}\label{3.09}
\lim_{m\to\infty}\mathbb{E}\Big[m^n\Delta_m^{i,n}q_U(x)1_{\{H(x)\geq 0\}}\Big]=\mathbb{E}\bigg[\frac{\partial^n}{\partial x_i^n} L(x)1_{\{L(x)\geq q_\alpha(x)\}}\bigg].
\end{equation}
For the second term of the right--hand side of equation \eqref{3.06}, by the quantile transformation lemma (equation \eqref{3.03}), we have
\begin{equation}\label{3.10}
\mathbb{E}\Big[m^n\Delta_m^{i,n}q_U(x)1_{C(x)}\Big]=\mathbb{E}\Big[m^n\Delta_m^{i,n}q_U(x)1_{C(x)\cap C_m^{i,n}(x)}\Big]+\mathbb{E}\bigg[m^n\Delta_m^{i,n}L(x)1_{C(x)\cap \overline{C_m^{i,n}(x)}}\bigg],
\end{equation}
where the set $C_m^{i,n}(x)$ is defined by $C_m^{i,n}(x):=\bigcap_{r=1}^n\big\{H\big(x+\frac{r}{m}e_i\big)\geq 0\big\}$. Due to equation \eqref{3.02}, the first expectation of equation \eqref{3.10} can be written by
\begin{equation}\label{3.11}
\mathbb{E}\Big[m^n\Delta_m^{i,n}q_U(x)1_{C(x)\cap C_m^{i,n}(x)}\Big]=m^n\Delta_m^{i,n}q_\alpha(x)\mathbb{P}\Big[C(x)\cap C_m^{i,n}(x)\Big].
\end{equation}
Since on the set $C(x)\subset \{L(x)=q_\alpha(x)\}$ the random variable $H(x)$ is a monotone increasing function with respect to $i$--th argument, for $r=1,\dots,n$, it holds 
\begin{equation}\label{3.12}
\bigg\{H\bigg(x+\frac{r}{m+1}e_i\bigg)\geq 0\bigg\}\subset \bigg\{H\bigg(x+\frac{r}{m}e_i\bigg)\geq 0\bigg\}~~~\text{for}~m=1,2,\dots.
\end{equation}
That means a set sequence $\big\{H\big(x+\frac{r}{m}e_i\big)\geq 0\big\}$ is monotone decreasing with respect to $m$ on the set $C(x)$. Thus, it holds
\begin{equation}\label{3.13}
\bigcap_{m=1}^\infty C_m^{i,n}(x)=\bigcap_{r=1}^n\bigcap_{m=1}^\infty\bigg\{H\bigg(x+\frac{r}{m}e_i\bigg)\geq 0\bigg\}=\big\{H(x)\geq 0\big\}.
\end{equation}
As a result, we have
\begin{equation}\label{3.14}
\lim_{m\to\infty}\mathbb{P}\Big[C(x)\cap C_m^{i,n}(x)\Big]=\mathbb{P}\big[C(x)\big]
\end{equation}
and a set sequence $\big\{\overline{C_m^{i,n}(x)}\big\}$ is monotone increasing and
\begin{equation}\label{3.15}
\lim_{m\to\infty}\mathbb{P}\Big[C(x)\cap \overline{C_m^{i,n}(x)}\Big]=0.
\end{equation}
Consequently, it follows from equations \eqref{3.01}, \eqref{3.11}, \eqref{3.14}, and \eqref{3.15} and Lemma \ref{l.01} that
\begin{eqnarray}
&&\lim_{m\to\infty}\mathbb{E}\Big[m^n\Delta_m^{i,n}q_U(x)1_{C(x)\cap C_m^{i,n}(x)}\Big]\nonumber\\
&&~=\frac{\partial^n}{\partial x_i^n}q_\alpha(x)\mathbb{P}\big[C(x)\big]=\frac{\partial^n}{\partial x_i^n}q_\alpha(x)\Big(\mathbb{P}\big[L(x)\geq q_\alpha(x)\big]-(1-\alpha)\Big)\label{3.16}
\end{eqnarray}
and
\begin{equation}\label{3.17}
\lim_{m\to\infty}\mathbb{E}\bigg[m^n\Delta_m^{i,n}L(x)1_{C(x)\cap \overline{C_m^{i,n}(x)}}\bigg]=0.
\end{equation}
Therefore, by equations \eqref{3.06}, \eqref{3.09} \eqref{3.10}, \eqref{3.16}, and \eqref{3.17} one conclude that
\begin{equation}\label{3.18}
\frac{\partial^n}{\partial x_i^n}\text{ES}_\alpha(x)=\frac{1}{1-\alpha}\bigg(\mathbb{E}\bigg[\frac{\partial^n}{\partial x_i^n}L(x)1_{\{L(x)\geq q_\alpha(x)\}}\bigg]-\frac{\partial^n}{\partial x_i^n}q_\alpha(x)\Big(\mathbb{P}\big[L(x)\geq q_\alpha(x)\big]-(1-\alpha)\Big)\bigg).
\end{equation}
That completes the proof of the Theorem.
\end{proof}

The following Corollary, corresponding to the absolute continuous portfolio loss random variable is a direct consequence of Proposition 2.
\begin{cor}\label{c.03}
Let $I\subset \mathbb{R}^d$ be an open set, for all $x\in I$, $H(x):=L(x)-q_\alpha(x)$ be an integrable absolute continuous random variable and its $n$--th order partial derivative $\frac{\partial^n}{\partial x_i^n}H(x)$ is continuous with respect to $i$--th ($i=1,\dots,d$) argument, and $Z:=\sup_{x\in I}\big|\frac{\partial^n}{\partial x_i^n}L(x)\big|$ be an integrable random variable. If on a set $C(x)$, $H(x)$ is a monotone increasing function with respect to $i$--th ($i=1,\dots,d$) argument $x_i$, then it holds
\begin{equation}\label{3.19}
\frac{\partial^n}{\partial x_i^n}\mathrm{ES}_\alpha(x)=\frac{1}{1-\alpha}\mathbb{E}\bigg[\frac{\partial^n}{\partial x_i^n}L(x)1_{\{L(x)\geq q_\alpha(x)\}}\bigg].
\end{equation}
\end{cor}

\begin{proof} Since $H(x)$ is the absolute continuous random variable, $\mathbb{P}\big[L(x)\geq q_\alpha(x)\big]=1-\alpha$. Therefore, equation \eqref{3.19} follows from equation \eqref{3.04}.
\end{proof}

\section{Conclusion}
The paper provides a Proposition that higher order derivatives of a generic nonlinear random variable preserve in a conditional expectation operator. Also, we consider results, which follows from Proposition for a generic absolute continuous random variable and discrete random variable. Further, the paper studies a derivative preserving condition for a positive homogeneous random variable. Finally, we consider conditions, which preserve higher order derivatives of the Expected Shortfall for a generic nonlinear portfolio loss random variable.

\bibliographystyle{apacite}
\bibliography{References}

\end{document}